\documentclass[12pt]{article}
\usepackage{amsthm}
\usepackage{amsmath,amsfonts}
\usepackage{geometry}
\usepackage{amssymb}
\usepackage{mathrsfs}
\usepackage{enumerate}
\usepackage{theoremref}
\usepackage{verbatim}
\newtheorem{thm}{Theorem}[section]
\newtheorem{lem}[thm]{Lemma}

\newtheorem{corollary}[thm]{Corollary}
\theoremstyle{definition}

\newtheorem{rmk}[thm]{Remark}

\newcommand{\om}{\Omega}
\newcommand{\bom}{\partial\om}

\newcommand{\rn}{\mathbb{R}^n}
\renewcommand{\div}{\textup{div}\,}
\newcommand{\vp}{\varphi}
\newcommand{\tr}{\textup{Tr}\,}
\newcommand{\wh}{\widehat}

\numberwithin{equation}{section}
\renewcommand{\epsilon}{\varepsilon}

\title{Dirichlet problems for second order linear elliptic equations with $L^{1}$-data}

\author{Hyunseok Kim \thanks{%
Department of Mathematics, Sogang University, Seoul, 121-742, Korea (kimh@sogang.ac.kr).}
\and Jisu Oh \thanks{%
Department of Mathematics, Sogang University, Seoul, 121-742, Korea (jisuoh@sogang.ac.kr); Present address: Department of Statistics, North Carolina State University, Raleigh, North Carolina, USA (joh26@ncsu.edu).}
 \thanks{The authors were supported by Basic Science
Research Program through the National Research Foundation of Korea (NRF) funded by
the Ministry of Education (No. NRF-2016R1D1A1B02015245).}
\date{}
}

\begin{document}

\maketitle

\begin{abstract}
We consider the Dirichlet problems for  second order linear elliptic equations in non-divergence and divergence forms on a bounded  domain $\Omega$ in $\rn$, $n \ge 2$:
$$
-\sum_{i,j=1}^n a^{ij}D_{ij} u + b \cdot  D u + cu = f \;\;\text{ in $\om$} \quad  \text{and} \quad  u=0 \;\;\text{ on $\partial \om$}
$$
and
$$
-  {\rm div} \left( A D  u \right) +  \div(ub) + cu = \div F  \;\;\text{ in $\om$} \quad  \text{and} \quad  u=0 \;\;\text{ on $\partial \om$} ,
$$
where $A=[a^{ij}]$ is symmetric, uniformly elliptic, and of vanishing mean oscillation (VMO).
The main purposes of this paper is to study unique solvability for both problems with $L^1$-data.  We   prove that if $\om$ is of class $C^{1}$, $ \div A + b\in L^{n,1}(\om;\rn)$, $c\in L^{\frac{n}{2},1}(\om) \cap L^s(\om)$  for some $1<s<\frac{3}{2}$, and $c\ge0$ in $\Omega$, then for each $f\in L^1 (\Omega )$, there exists a unique weak solution in $W^{1,\frac{n}{n-1},\infty}_0 (\om)$ of the first problem.
Moreover, under the additional condition that $\om$ is of class $C^{1,1}$ and $c\in L^{n,1}(\om)$, we show that for each $F \in L^1 (\om ; \rn )$, the second problem has a unique very weak solution in $L^{\frac{n}{n-1},\infty}(\om)$.
\end{abstract}

\section{Introduction}

Let $\Omega$ be a bounded domain in $\mathbb{R}^n$, where   $n\ge2$ is the dimension.
We study the following Dirichlet problems for second order linear elliptic equations in non-divergence and divergence forms:
\begin{equation}\label{Eq1}
\left\{
\begin{aligned}
-\sum_{i,j=1}^n a^{ij}D_{ij} u + b \cdot  D u + cu  &= f \quad \text{in $\om$}, \\
u &= 0 \quad  \text{on $\partial \om$}
\end{aligned}
\right.
\end{equation}
and
\begin{equation}\label{Eq3}
\left\{
\begin{aligned}
-{\rm div} \left( A D u \right) + \div(ub) + cu   &= \div F \quad \text{in $\om$}, \\
u&= 0 \,\,\,\,\qquad \text{on $\partial \om$}.
\end{aligned}
\right.
\end{equation}
Here  $A=[a^{ij}]:\rn \rightarrow \mathbb{R}^{n^2}$, $b=(b^1,\dots,b^n):\om \rightarrow \rn$, and $c:\om \rightarrow \mathbb{R}$ are locally integrable  functions.
We also consider the dual problems of (\ref{Eq1}) and (\ref{Eq3}):
\begin{equation}\label{Eq4}
\left\{
\begin{aligned}
-\sum_{i,j=1}^n D_{ij} \left( a^{ij} v \right) - \div(vb) + cv  &= \div G \quad \text{in $\om$}, \\
v &= 0 \,\,\,\,\qquad  \text{on $\partial \om$}
\end{aligned}
\right.
\end{equation}
and
\begin{equation}\label{Eq2}
\left\{
\begin{aligned}
-\div (A Dv ) - b \cdot D  v + cv  &= \div G \quad \text{in $\om$}, \\
v &= 0 \,\,\,\,\qquad  \text{on $\partial \om$}.
\end{aligned}
\right.
\end{equation}

Throughout the paper, we assume that $A = [a^{ij}]$   satisfies the following ellipticity and regularity conditions:
\begin{enumerate}
\item[(A1)] $A(x)$ is a symmetric matrix for each $x\in \rn$.
\item[(A2)]  There exists a constant $0<\delta<1$ such that
$\delta |\xi|^2 \le A (x)\xi \cdot \xi \le \delta^{-1} |\xi|^2$ for all $x, \xi\in \rn$.
\item[(A3)]  The functions  $a^{ij}$ are   of vanishing mean oscillation (VMO) for all $1\le i, j \le n$.
\end{enumerate}

\noindent Then a complete $L^p$-theory has been   developed for both    problems  \eqref{Eq1} and \eqref{Eq3}  when the lower order coefficients $b^i$ and $c$ are sufficiently regular and the domain $\Omega$ is smooth. For instance, if $b=0$ and $c=0$ identically,  or more generally if  $b, c$ are bounded and $c$ is nonnegative, then unique solvability in  $W^{2,p}(\Omega )$ of \eqref{Eq1} and  in $W^{1,p}(\Omega )$ of \eqref{Eq3} are  proved in \cite{AQ,BS,CFL,DK0,DK,gilbarg,krylov0}   for any $1<p<\infty$.
Recently, an $L^p$-theory was extended to  more general  $b$ and $c$ in critical Lebesgue or weak Lebesgue spaces. Precisely, it was shown by Krylov \cite{krylov} that if $b \in L^n (\Omega ; \rn)$, $c \in L^{n}(\om )$, and $c \ge 0$, then the problem \eqref{Eq1} is uniquely solvable in $ W^{2,p}(\Omega )$ for any $1< p<n$; that is, for any $f \in L^p (\Omega )$ with $1<p<n$, there exists a unique strong solution $u$ in $W_0^{1,p}(\Omega ) \cap W^{2,p}(\Omega )$ of \eqref{Eq1}. The corresponding $W^{1,p}$-results for \eqref{Eq3} have been proved only  for $1<p<n$ too. For more details, see the papers   \cite{kimkang,kimkim,kimkwon,kim,kwon0,kwon}.

It has been well known that an $L^p$-theory fails to hold for the limiting case when   $p=1$; indeed, there exists  $f\in L^1 (\Omega)$ such that the Poisson equation
\begin{equation}\label{Poisson-eq}
-\Delta u =f \quad\mbox{in}\,\, \Omega
\end{equation}
has no solutions in $W_{loc}^{2,1}(\Omega )$. As a simple example, we   take an $L^1$-function
$$
f(x)=\frac{1}{|x|^n \left| \log |x| \right|^{1+\epsilon}}
$$
defined on   $  \Omega = B_1  $, the open unit ball centered at the origin, where $0< \epsilon \le \frac{n-1}{n}$. Then the Poisson equation (\ref{Poisson-eq}) has a radial solution $v =v(r)$ satisfying
$$
-v''(r) -\frac{n-1}{r}v'(r) = \frac{1}{r^{n}| \log r | ^{1+\epsilon}}
$$
for $0<r<1$. Since
$$
|v'(r)| = \frac{1}{r^{n-1}} \int_0^r \frac{1}{t| \log t | ^{1+\epsilon}} dt \sim \frac{1}{\epsilon r^{n-1}| \log r |^{\epsilon}} \quad\mbox{as}\,\,r \to 0,
$$
it follows that     $|Dv| \in L^{\frac{n}{n-1},\infty}(B_\rho ) \setminus L^{\frac{n}{n-1}}(B_\rho  )$ for any $\rho < 1$. Now, if   $u$ is any solution of  (\ref{Poisson-eq}), then   $ u -v$  is harmonic and so smooth in $\Omega$. Hence any solution $u$ of  (\ref{Poisson-eq}) should satisfy $|Du| \in L^{\frac{n}{n-1},\infty}(B_\rho   ) \setminus L^{\frac{n}{n-1}}(B_\rho )$ for any $\rho< 1$. This implies,  by the Gagliardo-Nirenberg-Sobolev inequality (see e.g. \cite{evans})), that there exist no solutions in $W_{loc}^{2,1}(\Omega )$ of (\ref{Poisson-eq}). To find  a right regularity class for solutions of (\ref{Poisson-eq}) with  $f\in L^1(\om)$, let $w =   \Gamma \ast f$ be the Newtonian potential of $f$ over $\Omega$, where  $\Gamma$ denotes  the fundamental solution of the Laplace equation (see \cite{gilbarg}).  Then since $D \Gamma \in L^{\frac{n}{n-1},\infty}(\Omega ;\rn )$ and $\Gamma \in L^{p}(\Omega )$ for any $ p<\frac{n}{n-2}$, it follows from Young's convolution inequality in Lorentz spaces $L^{p,q}(\Omega )$ (see \thref{conv} in Section 2) that $w$ belongs to the Sobolev-Lorentz space
$ W^{1,\frac{n}{n-1},\infty}(\Omega  )$. Hence, if $u$ is any solution of  (\ref{Poisson-eq}), then   $u \in  W_{loc}^{1,\frac{n}{n-1},\infty}(\Omega  )$, that is, $u \in  W^{1,\frac{n}{n-1},\infty}(\Omega_0  )$ for any open set  $\Omega_0$ with $\overline{\Omega_0} \subset \Omega$.
It turns out that if $\Omega$ is a bounded $C^1$-domain, then for each  $f\in L^1(\om)$ there exists a unique weak solution $u  \in  W_0^{1,\frac{n}{n-1},\infty}(\Omega  )$  of (\ref{Poisson-eq}). Here  $W_0^{1,p,q}(\Omega )$ is defined as  the space of all functions in $W^{1,p,q}(\Omega )$ having zero trace on $\partial\Omega$. See Section 2 for more details on the Sobolev-Lorentz spaces $W^{k,p,q}(\Omega )$ and $W_0^{1,p,q}(\Omega )$, where $k \in \mathbb{N}$, $1 \le p<\infty$, and $1 \le q \le \infty$.

\medskip

The first purpose of this paper is to establish existence and uniqueness of weak solutions in $  W_0^{1,\frac{n}{n-1},\infty}(\Omega  )$  of the general elliptic problem \eqref{Eq1} for any  data $f$ in $L^1 (\Omega )$, which of course requires some additional assumptions on the  coefficients $A$,  $b$, and $c$.
First of all, in addition to (A1)-(A3), we need to assume at least that
$$
\div A\in L_{loc}^{1}(\om ;\rn ), \quad\mbox{that is}, \quad \sum_{i=1}^n D_i a^{ij} \in L_{loc}^{1}(\Omega) \quad\mbox{for each}\,\, j=1,\ldots , n ,
$$
which allows us to write  the equation in \eqref{Eq1}   as a divergence-form equation
\begin{equation}\label{Eq1-div}
-\div (A Du ) + \tilde{b} \cdot D  u + cu  = f \quad\mbox{in}\,\, \Omega ,
\end{equation}
where $\tilde{b} = \div A + b$.
Then by   a weak solution  in $W_0^{1,\frac{n}{n-1},\infty}(\Omega  )$ of  \eqref{Eq1} with $f$ belonging to  $L^1 (\Omega )$ or more generally $f$ being a distribution on $\Omega$, we mean a function $u \in W_0^{1,\frac{n}{n-1},\infty}(\Omega  )$ such that
\[
\tilde{b}  \cdot D u , c u \in L_{loc}^1 (\Omega )
\]
and
$$
\int_\om \left(A D  u \cdot  D  \vp + \vp \tilde{b}  \cdot  D  u  + cu \vp\right) dx=   \langle f , \vp \rangle
$$
for all $\vp \in C_c^\infty (\om)$. Note here that $|Du| \in L^{\frac{n}{n-1},\infty}(\Omega  )  $ and $u \in L^{p,\infty}(\Omega )$ for any finite $p \le \frac{n}{n-2}$, due to the Sobolev embedding theorem.   Hence in view of H\"{o}lder's inequality or the duality theorem in Lorentz spaces (see Lemmas \ref{holder} and \ref{dual-lorentz}), we are forced to impose that
\begin{equation}\label{btilde-c}
\tilde{b} \in L^{n,1}(\Omega ; \rn)\quad\mbox{and}\quad c \in L^{\frac{n}{2},1}(\Omega )\cap L^s (\Omega )\quad\mbox{for some}\,\,1 < s< \frac{3}{2}.
\end{equation}

We shall prove a    unique solvability result  in $  W_0^{1,\frac{n}{n-1},\infty}(\Omega  )$  for  the Dirichlet  problem  \eqref{Eq1} with   $f \in L^1 (\Omega )$ under the  condition (\ref{btilde-c}) on $A$, $b$, and $c$.
In fact, this is an easy corollary of the following result  for weak solutions  of the problems \eqref{Eq3} and \eqref{Eq2} in divergence form.

\begin{thm}\thlabel{thm1}
Let $\om$ be a bounded $C^{1}$-domain in $\rn$, $n \ge 2$. Assume   that
$b\in L^{n,1}(\om;\rn)$,
   $c\in L^{\frac{n}{2},1}(\om) \cap L^s(\om)$ for some $1<s<\frac{3}{2}$,   and $c\ge0$ in $\om$.
\begin{enumerate}[{\upshape (i)}]
\item For each $F\in L^{n,1}(\om;\rn)$, there exists a unique weak solution $u$ in $W^{1,n,1}_0(\om)$ of \eqref{Eq3}. Moreover,
$$
\|u\|_{W^{1,n,1}(\om)} \le C \|F\|_{L^{n,1}(\om;\rn)},
$$
where $C=C(n,\om,A,b,c)$.

\item For each $G\in L^{\frac{n}{n-1},\infty}(\om;\rn)$, there exists a unique weak solution $v$ in $W^{1,\frac{n}{n-1},\infty}_0(\om)$ of \eqref{Eq2}. Moreover,
$$
\|v\|_{W^{1,\frac{n}{n-1},\infty}(\om)} \le C \|G\|_{L^{\frac{n}{n-1},\infty}(\om;\rn)},
$$
where $C=C(n,\om,A,b,c)$.
\end{enumerate}
\end{thm}

Given  $f\in L^1(\om)$, define $F=D\Gamma * f$. Then
\[
\div F  = \Delta (\Gamma * f  ) = f  \quad\mbox{in}\,\, \om.
\]
Moreover, by Young's convolution inequality (\thref{conv}),
$$
\|F\|_{L^{\frac{n}{n-1},\infty}(\om; \rn)} \le C(n) \|f \|_{L^1 (\Omega )}.
$$
Therefore, as an immediate consequence of Theorem \ref{thm1} (ii), we obtain the following  result for the problem \eqref{Eq1} with $L^1$-data.

\begin{corollary}\thlabel{thm2}
Let $\om$ be a bounded $C^{1}$-domain in $\rn$, $n \ge 2$.
Assume that   $\div A  + b\in L^{n,1}(\om;\rn)$,
   $c\in L^{\frac{n}{2},1}(\om) \cap L^s(\om)$ for some $1<s<\frac{3}{2}$,   and $c\ge0$ in $\om$.
Then for each $f\in L^1(\om)$, there exists a unique weak solution $u$ in $W^{1,\frac{n}{n-1},\infty}_0(\om)$ of \eqref{Eq1}. Moreover,
$$
\|u\|_{W^{1,\frac{n}{n-1},\infty}(\om)} \le C \|f\|_{L^1(\om)},
$$
where $C=C(n,\om,A,b,c)$.
\end{corollary}

\begin{rmk}
It has been already shown in \cite{kimkang,kimkim,kimkwon,kwon} that if the lower order coefficients  $b$ and $c$ satisfy   $b \in L^n (\Omega;\rn )$,  $c\in L^{\frac{n}{2}}(\om) \cap L^s(\om)$ for some $1<s<\frac{3}{2}$, and $c \ge 0$ in $\Omega$, then the problem \eqref{Eq2} is uniquely solvable in  $W_0^{1,p}(\Omega )$ for  $\frac{n}{n-1}<p<\infty$, under the additional assumption  that $b \in L^r (\Omega;\mathbb{R}^2  )$ for some $r>2$ or  ${\rm div}\, A \in L^2 (\Omega ; \mathbb{R}^2 )$ if $n=2$. By real interpolation (see Theorem \ref{interpolation0}), the problem  \eqref{Eq2} is also solvable uniquely in  $W_0^{1,p,q}(\Omega )$ for $\frac{n}{n-1}<p<\infty$ and $1 \le q \le \infty$.
The novelty of Theorem \ref{thm1} is that the problem  \eqref{Eq2} is   solvable uniquely in  $W_0^{1,p,q}(\Omega )$ and  its dual problem  \eqref{Eq3} is   solvable uniquely in  $W_0^{1,p',q'}(\Omega )$ for one limiting case when  $p = \frac{n}{n-1}$ and $q = \infty$ if $b$ and $c$ satisfy the hypotheses of the theorem.
\end{rmk}

The second purpose of the paper is to study the problem (\ref{Eq3}) in divergence form for any $F$ in $L^1 (\Omega;\rn )$. To find a right class for solutions, we consider  the Poisson equation (\ref{Poisson-eq}) with $f =\div F$ for some $F\in L^1 (\Omega ; \rn )$. Then a solution of (\ref{Poisson-eq}) is given by
\[
w(x) = - (\Gamma * f)(x) = -\int_\om D\Gamma(x-y) \cdot F(y)\,dy.
\]
It follows from Young's convolution inequality  that $w \in L^{\frac{n}{n-1},\infty}(\om)$. Therefore, to solve the problem (\ref{Eq3}) for general $F\in L^1 (\Omega ; \rn )$, we need to introduce the notion of  very weak solutions in $  L^{\frac{n}{n-1},\infty}(\om)$, which is possible only for the adjoint or dual equation of a non-divergence form equation in \eqref{Eq1}. By a very weak solution  in $L^{\frac{n}{n-1},\infty}(\Omega  )$ of  \eqref{Eq4} with $G \in L^1 (\om;\rn )$, we mean a function $v \in L^{\frac{n}{n-1},\infty}(\Omega  )$ such that
\[
v |b| , c v \in L_{loc}^1 (\Omega )
\]
and
$$
\int_\om v \left( -\sum_{i,j=1}^n a^{ij} D_{ij}\psi +    b \cdot D \psi + c \psi \right) dx=   -\int_\om G \cdot D \psi \, dx
$$
for all $\psi\in C^{1,1}(\overline{\om})$ with $\psi =0$ on $\partial\Omega$.
Very weak solutions  in $L^{\frac{n}{n-1},\infty}(\Omega  )$ of  \eqref{Eq3} can be also defined if $A$ satisfies the additional condition $\div A  \in L_{loc}^{1}(\om;\rn)$.

We are   ready to state our result for very weak solutions  in $L^{\frac{n}{n-1},\infty}(\Omega  )$ of  \eqref{Eq3} with $F\in L^1 (\Omega ; \rn )$, which is an immediate consequence of the following result for the problem \eqref{Eq1} and its dual problem (\ref{Eq4}).

\begin{thm}\thlabel{thm1-2}
Let $\om$ be a bounded $C^{1,1}$-domain in $\rn$, $n \ge 2$. Assume   that
$b\in L^{n,1}(\om;\rn)$,
   $c\in L^{n,1}(\om)  $,   and $c\ge0$ in $\om$.
\begin{enumerate}[{\upshape (i)}]
\item For each $f\in L^{n,1}(\om)$, there exists a unique strong solution $u$ in $W^{2,n,1}(\om)\cap W^{1,n,1}_0(\om)$ of \eqref{Eq1}. Moreover,
$$
\|u\|_{W^{2,n,1}(\om)} \le C \|f\|_{L^{n,1}(\om)},
$$
where $C=C(n,\om,A ,b,\|c\|_{n,1;\om})$.

\item For each $G\in L^1(\om;\rn)$, there exists a unique very weak solution $v$ in $L^{\frac{n}{n-1},\infty}(\om)$ of $(\ref{Eq4})$.
Moreover,
$$
\|v\|_{L^{\frac{n}{n-1},\infty}(\om)} \le C \|G\|_{L^1(\om;\rn)},
$$
where $C=C(n,\om,A,b,\|c\|_{n,1;\om})$.
\end{enumerate}
\end{thm}

\begin{corollary}\thlabel{thm2-2}
Let $\om$ be a bounded $C^{1,1}$-domain in $\rn, n \ge  2$.
Assume that   $\div A  + b\in L^{n,1}(\om;\rn)$,
   $c\in L^{n,1}(\om) $,   and $c\ge0$ in $\om$.
 Then for each $F\in L^1(\om;\rn)$, there exists a unique very weak solution $u$ in $ L^{\frac{n}{n-1},\infty}(\om)$ of \eqref{Eq3}.
Moreover,
$$
\|u\|_{L^{\frac{n}{n-1},\infty}(\om)} \le C \|F\|_{L^1(\om;\rn)},
$$
where $C=C(n,\om,A,b,\|c\|_{n,1;\om})$.
\end{corollary}

\begin{rmk}
By real interpolation (see Theorem \ref{interpolation0}), it follows from the $W^{2,p}$-results in \cite{krylov} that if  $b \in L^n (\Omega;\rn )$,  $c\in L^{n}(\om)$, and $c \ge 0$ in $\Omega$, then the Dirichlet problem  \eqref{Eq1} is uniquely  solvable in  $W_0^{1,p,q}(\Omega )\cap W^{2,p,q}(\Omega )$ for $1<p<n$ and $1 \le q \le \infty$.
Theorem \ref{thm1-2} shows  that  the problem   \eqref{Eq1} is   solvable uniquely in  $W_0^{1,p,q}(\Omega )\cap W^{2,p,q}(\Omega )$   for one  limiting case when  $p = n$ and $q = 1$ if $b$ and $c$ satisfy the hypotheses of the theorem.
\end{rmk}

Under the hypotheses of  Theorem \ref{thm1}, we can  show  that the problems  \eqref{Eq3}   and \eqref{Eq2} are dual each other. Hence Part (ii) of  Theorem \ref{thm1} is deduced  from Part (i) by a  duality argument.
In Section 4,  we prove Part (i) of  Theorem \ref{thm1} by the   method of continuity. One of the key tools for the proof is the following estimate of Gerhadt-type \cite{gerhardt} in Lorentz spaces (see \thref{estimate2}):    for any $\epsilon>0$, there is a constant $C_\epsilon = C(n,\om, b,c,\epsilon)$ such that
$$
\|ub\|_{n,1;\om} + \|cu\|_{W^{-1,n,1}(\om)} \le \epsilon \|Du\|_{n,1;\om} + C_\epsilon \|u\|_{n,1;\om}
$$
for all $u \in W^{1,n,1}(\om)$. Using this estimate together with a priori estimates for the distribution functions  (Lemma \ref{apriori est-2-1}), we derive the a priori estimate  for   weak solutions in $W^{1,n,1}_0(\om)$ of \eqref{Eq3}, which allows us to prove   Part (i) of  \thref{thm1}. The proof of  \thref{thm1-2} is essentially  the same as that of  \thref{thm1} except for using the Aleksandrov maximum principle instead of   a priori estimates for the distribution functions.

The rest of the paper is organized as follows.
In Section 2,   Lorentz spaces and  Sobolev-Lorentz spaces are studied  in some details. We  recall  rather classical  results and then prove  some useful results on  real interpolation  and density for Sobolev-Lorentz spaces.
Using  theses results, we obtain  several Gerhardt-type estimates in Section 3.
Then Sections 4 and 5 are devoted to providing  complete proofs of Theorems \ref{thm1} and \ref{thm1-2}.

\section{Preliminaries}

We begin with introducing some standard notations.

For two quasi-Banach spaces $A$ and $B$ with $A\subset B$, we write  $A \hookrightarrow B$ if  $A$ is continuously embedded into $B$, that is, there is a constant $C$ such that $\|u\|_B \le C \|u\|_A$ for all $u\in A$. Also, if $A \hookrightarrow B$ and $B\hookrightarrow A$, we write $A=B$. We denote by $A^*$ the dual space of $A$. The dual pairing of $A^*$ and $A$  will be  denoted by  $\langle \cdot , \cdot \rangle$  without specifying the spaces for the sake of  simplicity.

\subsection{Lorentz spaces}
Let $(X,\mu)$ be a measure space.
For a measurable function $f :X \to \mathbb{R}$, we define its distribution function $d_f $  by
$$
 d_f(\alpha) = \mu \left( \{ x\in X : |f(x)|>\alpha \} \right)  \quad (\alpha >0 ).
 $$
The decreasing rearrangement of $f$ is defined on $(0,\infty)$ by
 $$
 f^*(t) = \inf \left\{ \alpha >0 : d_f (\alpha)\le t \right\}.
 $$

For $1 \le p ,q \le\infty$, we define the Lorentz space $L^{p,q}(X)$ as the set of all measurable functions $f$  on $X$ such that the quantity
\begin{equation*}\label{lorentz}
\|f \|_{p,q;X}=
\begin{cases}
 \displaystyle \left( \int _0 ^\infty [t^{1/p} f^{*}(t) ]^q\frac{dt}{t} \right )^{1/q}&\text{if}\,\, q<\infty,\\
 	 \quad \displaystyle \sup_{t>0} t^{1/p} f^{*}(t) &\text{if}\,\, q=\infty,
 \end{cases}
\end{equation*}
 is finite.
 Then $L^{p,q}(X)$ is a quasi-Banach space under the quasi-norm $\| \cdot \|_{p,q;X}$.
 In fact, $\|\cdot \|_{p,q;X}$ is normable when $1<p<\infty$ (see \cite[Section 4.4]{bennett} and \cite[Section 1.4]{Gra}, e.g.).
 Moreover, $L^{\infty,q}(X)=\{0\}$ for $1 \le q <\infty$ and $L^{p,q_1}(X) \hookrightarrow L^{p,q_2}(X)$ for $1 \le p\le \infty$ and $1\le q_1 \le q_2 \le \infty$.
 If $\mu(X)$ is finite, then $L^{p_1,\infty}(X) \hookrightarrow L^{p_2,1}(X)$ for $1 \le p_2<p_1 \le \infty$.
It is also well known that $L^{p,p}(X)$ coincides with the Lebesgue space $L^p(X)$ for $1\le p \le \infty$, that is, $\|\cdot\|_{p,p;X}$ is equivalent to the $L^p$-norm $\| \cdot \|_{p;X}$.
It is easy to check that
\begin{equation}\label{lorentz-exp}
\left\| |f|^\alpha \right\|_{p,q;X} = \left\|f \right\|_{\alpha p, \alpha q;X}^\alpha
\end{equation}
for all $1\le p \le\infty$, $1\le q\le\infty$, and $\alpha\ge \max(1/p,1/q)$.

\medskip

It is standard (see, e.g., \cite[Theorem 5.3.1]{bergh}) that a Lorentz space $L^{p,q}(X)$ is the real interpolation space of two Lebesgue spaces.

\begin{lem}\thlabel{itp-lorentz}
Let $1\le p,p_1, p_2 \le \infty$ and $0<\theta<1$ satisfy
$$
p_1 \neq p_2 \quad \text{and} \quad \frac{1}{p} = \frac{1-\theta}{p_1} + \frac{\theta}{p_2}.
$$
Then for any $ 1\le q\le \infty$,
$$
\left( L^{p_1}(X) , L^{p_2}(X)\right)_{\theta,q} = L^{p,q}(X).
$$
\end{lem}

By \thref{itp-lorentz}, H\"{o}lder's inequality and Young's convolution inequality can be   extended to Lorentz spaces as follows for some limiting cases:

\begin{lem}\thlabel{conv}
Let $1<p<\infty$ and $1 \le q \le \infty$.
\begin{enumerate}
\item[\textup{(i)}] If $f\in L^{p,q}(X)$ and $g\in L^\infty (X)$, then $\|f   g \|_{p,q;X} \le C(p,q) \|f\|_{p,q;X} \|g\|_{\infty;X}$.
\item[\textup{(ii)}] If $f\in L^{p,q}(X)$ and $g\in L^1(X)$, then $\|f \ast g \|_{p,q;X} \le C(p,q) \|f\|_{p,q;X} \|g\|_{1;X}$.
\end{enumerate}
\end{lem}

The following is the general H\"older inequality for Lorentz spaces.
\begin{lem}\thlabel{holder}
Let $1\le p,p_1,p_2<\infty$ and $1\le q,q_1,q_2 \le \infty$ satisfy
$$
\frac{1}{p} = \frac{1}{p_1}+\frac{1}{p_2} \quad\text{and}\quad \frac{1}{q}\le \frac{1}{q_1}+\frac{1}{q_2}.
$$
Then for any $f\in L^{p_1,q_1}(X)$ and $g\in L^{p_2,q_2}(X)$,
$$
\|fg\|_{p,q;X}\le C \|f\|_{p_1,q_1;X} \|g\|_{p_2,q_2;X},
$$
where $C=C(p_1,p_2,q,q_1,q_2)$.
\end{lem}
\begin{proof}
This lemma was proved by O'neil \cite[Theorem 3.4]{oneil} for $p>1$.
Suppose that $p=1$.
Then choosing any $\alpha>1$, we have
\begin{align*}
\|fg \|_{1.q;X} &= \| |f|^{1/\alpha} |g|^{1/\alpha}\|^\alpha_{\alpha,\alpha q;X}\\
	 &\le C \left(\||f|^{1/\alpha} \|_{\alpha p_1 ,\alpha q_1;X}\|| g|^{1/\alpha} \|_{\alpha p_2 ,\alpha q_2;X} \right)^\alpha
	 	 = C \|f\|_{p_1,q_1;X} \|g\|_{p_2,q_2;X}. \tag*{\qedhere}
\end{align*}
\end{proof}

For $1\le p \le\infty$, let $p'= \frac{p}{p-1}$ denote the H\"older conjugate of $p$.
The following is a  duality result for Lorentz spaces (see, e. g., \cite[Theorem 1.4.17]{Gra}):

\begin{lem}\thlabel{dual-lorentz}
Suppose that $(X,\mu)$ is a nonatomic $\sigma$-finite measure space. Then for $1<p<\infty$ and $1\le q < \infty$, we have
$$
L^{p,q}(X)^* = L^{p',q'}(X).
$$
\end{lem}

\subsection{Sobolev-Lorentz spaces}

Let $\om$ be any domain in $\rn  $  and let $k \in \mathbb{N}$.
For $(p,q)=(1,1)$ or $1 < p \le \infty$, $1\le q \le \infty$, the Sobolev-Lorentz space $W^{k,p,q}(\om)$ is defined as the space  of all $u\in L^{p,q}(\om)$ such that $D^\alpha u$ exists weakly and belongs to $L^{p,q}(\om)$ for all multi-indices $\alpha$ with $|\alpha|\le k$.
Then  $W^{k,p,q}(\om)$ is a quasi-Banach space equipped with  the quasi-norm
$$
\| u\|_{W^{k,p,q}(\om)} = \sum_{|\alpha|\le k} \|D^\alpha u\|_{p,q;\om}.
$$
We denote by $\widehat{W}^{k,p,q}(\om)$ the closure of $C_c^\infty(\om)$ in $W^{k,p,q}(\om)$. For  $1<p<\infty$, we define $W^{-k,p,q}(\om)$ as the dual space of $\widehat W^{k,p',q'}(\om)$. Note that  $W^{k,p,p}(\om)$ coincides with the usual Sobolev space $W^{k,p}(\om)$.
This motivates us to  write $\widehat W^{k,p}(\om) = \widehat W^{k,p,p}(\om)$  and  $W^{-k,p}(\om)=W^{-k,p,p}(\om)$.

The following is a fundamental  embedding theorem for Sobolev-Lorentz spaces.

\begin{lem}[\upshape{\cite[Theorem 7]{adams}}]\thlabel{ebd}
Let $\om$ be a bounded Lipschitz domain in $\rn$.
\begin{enumerate}
\item[\textup{(i)}] If $1<p<n$ and $1\le q \le\infty$, then $W^{1,p,q}(\om)\hookrightarrow L^{p^*,q}(\om)$, where $p^* =\frac{np}{n-p}$ denotes the Sobolev conjugate of $p$.
\item[\textup{(ii)}]  If $k\in\mathbb{N}$ and $1<p<\infty$ satisfy $kp=n$, then $W^{k,p,1}(\om)\hookrightarrow C(\overline \om)$.
\end{enumerate}
\end{lem}

A Sobolev-Lorentz space is the real interpolation space of two Sobolev spaces of the same differentiability order, as shown by the  following important result essentially  due to DeVore and Scherer    \cite{devore}.

\begin{lem}\thlabel{interpolation}
Let $\om$ be the whole space $\rn$ or a bounded Lipschitz domain in $\rn$.
If $1 < p,p_1,p_2 < \infty$ and $0<\theta<1$ satisfy
$$
p_1 \neq p_2 \quad \text{and} \quad \frac{1}{p}=\frac{1-\theta}{p_1}+\frac{\theta}{p_2},
$$
then for any $k\in \mathbb{N}$ and $1\le q \le\infty$,
\begin{equation*}
\left(W^{k,p_1}(\om), W^{k,p_2}(\om) \right)_{\theta,q} = W^{k,p,q}(\om).
\end{equation*}
\end{lem}
\begin{proof}
In fact, DeVore and Scherer   \cite[Theorem 2]{devore} proved that
\begin{equation*}
\left(W^{k,1}(\om), W^{k,\infty}(\om) \right)_{\theta,q}= W^{k,\frac{1}{1-\theta},q}(\om).
\end{equation*}
To prove the lemma, we define
$$
\theta_1=1-\frac{1}{p_1} \quad\text{and}\quad \theta_2 = 1-\frac{1}{p_2}.
$$
Then by the reiteration theorem (see, e.g., \cite[Theorem 3.5.3]{bergh}), we get
\begin{align*}
\left(W^{k,p_1}(\om),W^{k,p_2}(\om)\right)_{\theta,q}& = \left((W^{k,1}(\om),W^{k,\infty}(\om))_{\theta_1 ,p_1},(W^{k,1}(\om),W^{k,\infty}(\om))_{\theta_2 ,p_2}\right)_{\theta,q}\\
	& =\left(W^{k,1}(\om),W^{k,\infty}(\om)\right)_{(1-\theta)\theta_1 + \theta \theta_2,q} = W^{k,p,q}(\om). \tag*{\qedhere}
\end{align*}
\end{proof}

\begin{lem}\thlabel{extension}
Let $\om$ be a bounded Lipschitz domain in $\rn$. Then there exists a linear operator $E$ from $L^1(\om)$ into $L^1(\rn)$ such that
\begin{enumerate}
\item[\textup{(i)}] $Eu = u$ in $\om$ for all $u\in L^1(\om)$ and
\item[\textup{(ii)}] $E$ is bounded from $W^{k,p,q}(\om)$ into $W^{k,p,q}(\rn)$ for all $k\in \mathbb{N}$, $1<p<\infty$, and $1\le q \le\infty$.
\end{enumerate}

\end{lem}
\begin{proof}
By \cite[Theorem 6.5]{stein}, there is a linear operator $E$ from $L^1(\om)$ into $L^1(\rn)$ such that $Eu=u$ in $\om$ for all $u\in L^1(\om)$ and $E$ is bounded from $W^{k,r}(\om)$ into $W^{k,p}(\rn)$ for all $k\in\mathbb{N}$ and $1\le p\le\infty$. Hence the lemma   immediately follows  from \thref{interpolation} by real interpolation.
\end{proof}

Let $\om$ be a bounded Lipschitz domain in $\rn$. Then there is a unique bounded linear operator $\tr$ from $W^{1,1}(\om)$ to $L^1(\partial \om)$ such that $\textup{Tr}\, u = u |_{\partial \om}$ for all $u \in C^\infty(\overline{\om})$. Moreover, the trace operator $\tr$ is bounded from $W^{1,p}(\om)$ into $L^p(\bom)$ for all $1\le p<\infty$.
Hence it follows from Lemmas \ref{itp-lorentz} and \ref{interpolation} that $\tr$ is bounded from $W^{1,p,q}(\om)$ into $L^{p,q}(\bom)$ for all $1<p<\infty$ and $1\le q \le \infty$.
For $(p,q)=(1,1)$ or $1<p\le \infty$, $1\le q\le \infty$, we define $W^{1,p,q}_0(\om)$ as the space of all $u \in W^{1,p,q}(\om)$ such that $\tr u=0$ on $\partial \om$.
In particular, we define $W^{1,p}_0(\om) = W^{1,p,p}_0(\om)$ for $1\le p \le \infty$.
Then it is well known that $W^{1,p}_0(\om)=\wh W^{1,p}(\om)$ for $1\le p<\infty$.

Consider the following Dirichlet problems for elliptic equations with no lower order terms:
\begin{equation}\label{Poisson equation}
\left\{
\begin{aligned}
-\sum_{i,j=1}^n a^{ij}D_{ij} u &= f \quad \text{in $\om$}, \\
u &= 0 \quad \text{on $\partial \om$}
\end{aligned}
\right.
\end{equation}
and
\begin{equation}\label{Poisson equation2}
\left\{
\begin{aligned}
-\div  ( A D  u  ) &= \div F \quad \text{in $\om$}, \\
u &= 0 \,\,\,\,\qquad  \text{on $\partial \om$} ,
\end{aligned}
\right.
\end{equation}
where $A =[a^{ij}]$ satisfies the hypotheses (A1), (A2), and (A3) in Introduction.

The following    is a unique solvability   result in Sobolev-Lorentz spaces for  the problems   \eqref{Poisson equation} and \eqref{Poisson equation2}
with data in the Lorentz spaces $L^{p,q}$.

\begin{lem}\thlabel{kimtsai}
Let $1<p<\infty$ and $1\le q \le \infty$.
\begin{enumerate}[{\upshape(i)}]
\item Let $\om$ be a bounded $C^1$-domain in $\mathbb{R}^n  $. Then for each $F\in L^{p,q}(\om;\mathbb{R}^n)$, there exists a unique weak solution $u\in W^{1,p,q}_0(\om)$ of \eqref{Poisson equation2}.
Moreover,
$$
\|u \|_{W^{1,p,q}(\om)}\le C\|F\|_{p,q;\om},
$$
where $C=C(n,\om,A,p,q)$.
\item Let $\om$ be a bounded $C^{1,1}$-domain in $\rn  $. Then for each  $f\in L^{p,q}(\om)$, there exists a unique strong solution $u\in W^{2,p,q}(\om)\cap W^{1,p,q}_0(\om)$ of \eqref{Poisson equation}.
Moreover,
$$
\|u \|_{W^{2,p,q}(\om)}\le C\|f\|_{p,q;\om},
$$
where $C=C(n,\om,A,p,q)$.
\end{enumerate}
\end{lem}
\begin{proof} By a standard real interpolation argument based on    Lemmas \ref{itp-lorentz} and \ref{interpolation}, it suffices to prove the lemma for the special case when $1< p=q< \infty$ (see the proofs  of \cite[Propositions 3.13 and 3.14]{kim} for details).

Suppose that $1< p=q< \infty$. Then   the lemma was already proved by Auscher-Qafsaoui \cite{AQ} and  Chiarenza-Frasca-Longo \cite{CFL} (see also \cite{CaTr}), except for showing  that  the constant $C$ in Part (ii) depends only on $n$, $\Omega$, $A$, and $p=q$. This explicit dependence of $C$ was shown by Krylov \cite[Theorem 4.2]{krylov} if $1<p<n$. Suppose now that $n \le p < \infty$, and   let $u\in W^{2,p}(\om)\cap W^{1,p}_0(\om)$ be a strong solution of \eqref{Poisson equation} with $f\in L^p (\Omega )$.
Then by \cite[Theorem 8]{DK}, there exists  a constant  $C=C(n,\om,A,p)$ such that
$$
\|u \|_{W^{2,p}(\om)}\le C \left( \|f\|_{p;\om} + \|u\|_{p;\om} \right).
$$
Choose any $p_0$ with $n/2 < p_0 < n $. Then by the Sobolev embedding theorem and \cite[Theorem 4.2]{krylov}, we have
\[
\|u\|_{p;\om}   \le C \|u\|_{W^{2,p_0} (\om)} \le C   \|f\|_{p_0 ;\om} \le C   \|f\|_{p;\om} ,
\]
where $C=C(n,\om,A,p)$. This completes the proof of the lemma.
\end{proof}

\thref{kimtsai} can be used to prove the following interpolation results.

\begin{thm}\thlabel{interpolation0} Let $1<p,p_1,p_2<\infty$ and $0<\theta<1$ satisfy
$$
p_1 \neq p_2 \quad \text{and} \quad \frac{1}{p}=\frac{1-\theta}{p_1}+\frac{\theta}{p_2}.
$$

\begin{enumerate}[{\upshape(i)}]
\item Let $\om$ be a bounded $C^1$-domain in $\mathbb{R}^n  $. Then for any $1\le q \le \infty$,
\begin{equation}\label{interpolation0-1}
W^{1,p,q}_0(\om)= \left (W^{1,p_1}_0(\om),W^{1,p_2}_0(\om) \right)_{\theta,q}
\end{equation}
and
\begin{equation}\label{interpolation0-2}
W^{-1,p,q}(\om)= \left (W^{-1,p_1}(\om),W^{-1,p_2}(\om) \right)_{\theta,q}.
\end{equation}
\item Let $\om$ be a bounded $C^{1,1}$-domain in $\rn  $. Then  for any $1\le q \le \infty$,
\begin{equation}\label{interpolation0-12}
W^{1,p,q}_0(\om) \cap W^{2,p,q} (\om)= \left (W^{1,p_1}_0(\om) \cap W^{2,p_1} (\om),W^{1,p_2}_0(\om)\cap W^{2,p_2} (\om) \right)_{\theta,q}.
\end{equation}
\end{enumerate}
\end{thm}

\begin{proof} Assume that $\om$ is  a bounded $C^1$-domain in $\mathbb{R}^n$.

Suppose that $u\in \left(W_0^{1,p_1}(\om),W_0^{1,p_2}(\om)\right)_{\theta,q}$. Then by \thref{interpolation},
$$
u\in \left(W^{1,p_1}(\om),W^{1,p_2}(\om) \right)_{\theta,q}=W^{1,p,q}(\om).
$$
 Since $u$ has trivial boundary trace, we see that $u\in W^{1,p,q}_0(\om)$.
 To prove the converse, we may assume that $p_1<p_2$.
 Let $T:L^{p_1}(\om;\mathbb{R}^n)\rightarrow W_0^{1,p_1}(\om)$  be the solution operator of \eqref{Poisson equation2};
 that is, for every $F\in L^{p_1}(\om;\rn)$,
$u=T(F)$ is a unique weak solution in $W_0^{1,p_1}(\om)$ of \eqref{Poisson equation2}.
 The existence of $T$ is guaranteed by \thref{kimtsai} (i).
 Moreover,  \thref{kimtsai} (i) shows that $T$ is bounded from $L^{p_i}(\om;\rn)$ into $W^{1,p_i}_0(\om)$ for $i=1,2$.
 Therefore, by \thref{itp-lorentz}, $T$ is bounded from $L^{p,q}(\om;\rn)$ into $\left(W_0^{1,p_1}(\om),W_0^{1,p_2}(\om)\right)_{\theta,q}$.
 Now suppose that $u\in W^{1,p,q}_0(\om)$ and $F=-A Du$.
 Then $F\in L^{p,q}(\om;\rn)$  and $u$ is a weak solution of \eqref{Poisson equation2}.
 By the uniqueness of weak solutions in $W^{1,p_1}_0(\om)$ of $\eqref{Poisson equation2}$, we conclude that $u=T(F)\in \left(W_0^{1,p_1}(\om),W_0^{1,p_2}(\om)\right)_{\theta,q}$.
 This proves \eqref{interpolation0-1}.

 To prove \eqref{interpolation0-2}, we recall that $W^{-1,r}(\om)=W^{1,r'}_0(\om)^*$ for $1<r<\infty$.
 Then \eqref{interpolation0-2} can be deduced from \eqref{interpolation0-1} by the duality theorem (see, e.g., \cite[Theorem 3.7.1]{bergh}) in real interpolation theory.
 Indeed, if we denote by $\left(W^{1,p_1'}_0(\om) , W^{1,p_2'}_0(\om)\right)^0_{\theta,q'}$   the completion of $C_c^\infty(\om)$ in $\left( W^{1,p_1'}_0(\om), W^{1,p_2'}_0(\om)\right)_{\theta,q'}$, then
 \begin{align*}
 \left( W^{-1,p_1}(\om), W^{-1,p_2}(\om)\right)_{\theta,q} &= \left( W^{1,p'_1}_0(\om)^* , W^{1,p_2'}_0(\om)^*\right)_{\theta,q} \\
 	& = \left[ \left( W^{1,p_1'}_0(\om), W^{1,p_2'}_0(\om)\right)_{\theta,q'}^0 \right]^*.
 \end{align*}
 It was already shown that $\left(W^{1,p_1'}_0(\om),W^{1,p_2'}_0(\om)\right)_{\theta,q'}= W^{1,p',q'}_0(\om)$.
 Therefore,
 $$
 \left(W^{-1,p_1}(\om) , W^{-1,p_2}(\om)\right)_{\theta,q} = \widehat W^{1,p',q'}(\om)^* = W^{-1,p,q}(\om).
 $$

 Assume in addition that $\om$ is  of class $C^{1,1}$. Then the proof of (\ref{interpolation0-12}) is exactly the same as \eqref{interpolation0-1}, except for using \thref{kimtsai} (ii) instead of \thref{kimtsai} (i).
\end{proof}

\begin{corollary}\thlabel{kimtsai2}
Let $\om$ be a bounded $C^1$-domain in $\rn  $. If $1<p<\infty$ and $1\le q\le \infty$, then
for each $f\in W^{-1,p,q}(\om)$, there is a unique weak solution $u\in W^{1,p,q}_0(\om)$ of \eqref{Poisson equation2} with $f$   in place of $\div F$. Moreover,
$$
\|u\|_{W^{1,p,q}(\om)} \le C \|f\|_{W^{-1,p,q}(\om)},
$$
where $C=C(n,\om,A,p,q)$.
\end{corollary}
\begin{proof}
For the special case when $1<p=q<\infty$, the desired $W_0^{1,p}$-result immediately follows from \thref{kimtsai} (i) since every $f\in W^{-1,p}(\om)$ can be written as $f= {\rm div}\, F$ for some $F\in L^p (\om ; \rn )$. Then the corollary in its full generality is easily deduced from the $W_0^{1,p}$-results by real interpolation based on  \thref{interpolation0}.
\end{proof}

The following density results will be used to prove the main theorems of  the paper later in Sections 4 and 5.

\begin{lem}\thlabel{dense}
Let $1<p<\infty$ and $1\le q<\infty$.
\begin{enumerate}[{\upshape(i)}]
\item Let $\om$ be any  domain in $\rn$. Then $C^\infty_c(\om)$ is dense in $L^{p,q}(\om)$.
\item Let $\om$ be a bounded Lipschitz domain in $\rn$. Then for each $u\in W^{1,p}_0(\om) \cap L^\infty(\om)$ there exists  a sequence $\{u_k\}$ in $C_c^\infty(\om)$ such that $u_k \rightarrow u$ in $W^{1,p}(\om)$ and $u_k \rightarrow u$ weakly-$*$ in $L^\infty(\om)$.
\item Let $\om$ be a bounded $C^1$-domain in $\rn  $. Then $C^\infty_c(\om)$ is dense in $W^{1,p,q}_0(\om)$.
\item Let $\om$ be a bounded $C^{1,1}$-domain in $\rn  $. Then for each $u \in W^{1,p,q}_0(\om) \cap W^{2,p,q}(\om )$ there exists a sequence $\{u_k\}$ in $C^{1,1}(\overline{\om})$ such that $ u_k  =0$ on $\partial\Omega$ and  $u_k \rightarrow u$ in $W^{2,p,q}(\om)$.
\end{enumerate}
\end{lem}
\begin{proof}
(i) Since $C_c^\infty(\om)$ is dense in $L^{r}(\om)$ for $1<r<\infty$, it follows from \thref{itp-lorentz} and a standard density theorem (see, e.g., \cite[Theorem 3.4.2]{bergh}) that $C^\infty_c(\om)$ is dense in $L^{p,q}(\om)$.

(ii) We closely follow the proof of \cite[Theorem 5.5.2]{evans}.
Using partitions of unity and flattening out $\partial \om$, we may assume that $u\in W^{1,p}_0(\rn_+)\cap L^\infty(\rn_+)$ and $u$ has compact support in $\overline {\rn_+}$.
Then since $\tr u=0$ on $\partial \rn_+=\mathbb{R}^{n-1}$, we have
\begin{equation}\label{dense-1}
\int_{\mathbb{R}^{n-1}} |u(x',x_n)|^p dx' \le x_n^{p-1} \int_0^{x_n}\int_{\mathbb{R}^{n-1}}|D_n u(x',t)|^p dx'dt
\end{equation}
for all $x_n\ge0$.
Choose a function $\zeta \in C_c^\infty(\mathbb{R})$ such that $0\le \zeta\le1$, $\zeta(t)=0$ for $|t|\ge4$, and $\zeta=1$ on $[0,2]$.
Define $u_k(x)= u(x) \left( 1-\zeta(kx_n) \right)$ on $\rn_+$.
Then it easily follows from \eqref{dense-1} that $u_k \rightarrow u$ in $W^{1,p}(\rn_+)$.
Observe that $u_k=0$ on $\{0<x_n<2/k\}$, $|u_k| \le |u|$, and $u\in L^\infty(\rn_+)$.
Hence if $\rho$ is a standard mollifier and $\rho_k(x)=k^n \rho(kx)$, then $u_k \ast \rho_k \rightarrow u$ in $W^{1,p}(\rn_+)$ and $u_k \ast \rho_k \rightarrow u$ weakly-$*$ in $L^\infty(\rn_+)$.

(iii) Since $C_c^\infty(\om)$ is dense in $W^{1,r}_0(\om)$ for $1<r<\infty$, it follows from \thref{interpolation0} and   \cite[Theorem 3.4.2]{bergh}  that $C^\infty_c(\om)$ is dense in $W^{1,p,q}_0(\om)$.

(iv) By \thref{interpolation0}, it suffices to show that $\left\{ u\in C^{1,1}(\overline{\om}): u =0 \,\,\mbox{on}\,\,\partial\Omega \right\}$ is dense in $W^{1,p}_0(\om) \cap W^{2,p}(\om )$. Assume that $u\in W^{1,p}_0(\rn_+)\cap  W^{2,p}(\rn_+)$ and $u$ has compact support in $\overline {\rn_+}$.
Let $\overline{u}$ be the extension of $u$ to $\rn$ by odd refection:
\[
\overline{u}(x' , x_n ) = \left\{
\begin{aligned}
u(x' , x_n ) \quad \text{if } x_n \ge 0, \\
-u (x', x_n )  \quad \text{if } x_n < 0 .
\end{aligned}
\right.
\]
Then  $\overline{u}\in  W^{2,p}(\rn )$ and ${\rm Tr} \, u=0$ on $\partial \rn_{+}$. Now, if $u_k$ is the restriction of $\overline{u} \ast \rho_k$ to $\rn_{+}$, then $u_k \in C^\infty (\overline {\rn_+} )$, $u_k =0$ on $\partial\rn_{+}$,  and $u_k \to u$ in $W^{2,p}(\rn_+)$. The proof is complete by  using partitions of unity.
\end{proof}

\section{Gerhardt-type estimates in  Lorentz spaces}
Throughout this section, we assume that $\om $ is a bounded Lipschitz domain in $\rn $, where $n \ge 2$. Then we shall prove several   estimates of Gerhardt-type \cite{gerhardt} that will be used crucially to derive   a priori estimates for solutions of the problems \eqref{Eq1}, \eqref{Eq3},  and \eqref{Eq2}.

First we state Gagliardo-Nirenberg inequalities in Sobolev-Lorentz spaces.

\begin{thm}\thlabel{GN}
Let $k,k_1,k_2\in \mathbb{N}$, $1<p,p_1,p_2<\infty$, $1\le q,q_1,q_2\le\infty$, and $0<\theta<1$ satisfy
$$
(1-\theta)k_1 + \theta k_2 -k \ge n\left(\frac{1-\theta}{p_1} + \frac{\theta}{p_2} -\frac{1}{p}\right) \ge 0 \quad \text{and} \quad \frac{1}{q}\le \frac{1-\theta}{q_1}+\frac{\theta}{q_2}.
$$
Then for all $u\in W^{k_1,p_1,q_1}(\om) \cap W^{k_2,p_2,q_2}(\om)$,
$$
\|u\|_{W^{k,p,q}(\om)} \le C \|u\|_{W^{k_1,p_1,q_1}(\om)}^{1-\theta}\|u\|_{W^{k_2,p_2,q_2}(\om)}^\theta,
$$
where $C=C(n,\om,k,k_1,k_2,p,p_1,p_2,q,q_1,q_2,\theta)$.
\end{thm}
\begin{proof}
It was shown in \cite[Theorem 5.9]{itp} that the theorem holds when $\om=\rn$.
If $\om$ is a bounded Lipschitz domain in $\rn$, then
\begin{align*}
\|u\|_{W^{k,p,q}(\om)} &\le C \|Eu\|_{W^{k,p,q}(\rn)} \\ &\le C \|Eu\|_{W^{k_1,p_1,q_1}(\rn)}^{1-\theta}\|Eu\|_{W^{k_2,p_2,q_2}(\rn)}^\theta\\
	& \le C \|u\|_{W^{k_1,p_1,q_1}(\om)}^{1-\theta}\|u\|_{W^{k_2,p_2,q_2}(\om)}^\theta
\end{align*}
for all $u\in W^{k_1,p_1,q_1}(\om) \cap W^{k_2,p_2,q_2}(\om)$, where $E$ is the extension operator in \thref{extension}.
\end{proof}

\begin{corollary}\thlabel{GN2}
Let $1<p<\infty$ and $1\le q\le\infty$. Then for any $\varepsilon >0$, there is a constant $C_\epsilon=C(n,\om, p,q,\epsilon)$ such that
\[
\|u\|_{W^{1,p,q}(\om)} \le \varepsilon \|D^2u\|_{p,q;\om}+C_\epsilon \|u\|_{p,q;\om}
\]
for all $u\in W^{2,p,q}(\om)$.
\end{corollary}
\begin{proof}
Suppose that $u\in W^{2,p,q}(\om)$. Let $\epsilon >0$ be given. Then by  \thref{GN}, we have
\begin{align*}
\|u\|_{W^{1,p,q}(\om)} &\le C \|u\|_{W^{2,p,q}(\om)}^{1/2}\|u\|_{p,q;\om}^{1/2}\\
	& \le C\|D^2u\|_{p,q;\om}^{1/2}\|u\|_{p,q;\om}^{1/2} + C\|Du\|_{p,q;\om}^{1/2}\|u\|_{p,q;\om}^{1/2}  + C\|u\|_{p,q;\om}\\
	&\le\epsilon\|D^2u\|_{p,q;\om} + \frac{1}{2} \|Du\|_{p,q;\om} +C_\varepsilon \|u\|_{p,q;\om},
\end{align*}
which completes the proof.
\end{proof}

\begin{lem}\thlabel{estimate1}
Assume that $b\in L^{n,1}(\om;\rn)$ and $c\in L^{n,1}(\om)$. Then
for any $\epsilon>0$, there is a constant $C_\epsilon=C(n,\om,b,\|c\|_{n,1;\om}, \epsilon)$ such that
\begin{equation*}
 \| b \cdot Du \|_{n,1;\om} + \|c u\|_{n,1;\om} \le \epsilon \| D^2 u \|_{n,1;\om} + C_\epsilon \| u \| _{n,1;\om}
\end{equation*}
for all $u\in W^{2,n,1}(\om)$.
\end{lem}

\begin{proof} By Lemma  \ref{dense}, there is a sequence  $\{b_k \}$ in $C_c^{\infty}(\om;\rn)$ such that $b_k\rightarrow b$ in $L^{n,1}(\om;\rn)$. By Lemmas \ref{conv} and \ref{ebd}, we have
\begin{align*}
\|b \cdot  Du\|_{n,1;\om} &= \|(b -b_{k} )\cdot Du+b_{k}\cdot Du\|_{n,1;\om} \nonumber\\
 &\le C\|b-b_k\|_{n,1;\om}\|Du\|_{\infty;\om} + C\|b_k\|_{\infty;\om} \|Du\|_{n,1;\om} \nonumber\\
 &\le C\|b-b_k\|_{n,1;\om}\|Du\|_{W^{1,n,1}(\om)} + C\|b_k\|_{\infty;\om} \|Du\|_{n,1;\om} \nonumber\\
 &\le C_* \|b-b_k\|_{n,1;\om}\|D^2u\|_{n,1;\om}+ C(\|b-b_k\|_{n,1;\om}+\|b_k\|_{\infty;\om})\|Du\|_{n,1;\om} ,
 \end{align*}
where $C_* =C_* (n,\om)$. Now, given $\epsilon>0$, we choose a sufficiently large $k$ such that
$$
C_* \|b-b_k\|_{n,1;\om}\le\frac{\epsilon}{4}.
$$
Then by \thref{GN2}, there is a constant $C_\epsilon=C(n,\om, b,\epsilon)$ such that
\begin{equation*}
\|b \cdot  Du\|_{n,1;\om} \le \frac{\epsilon}{2}\|D^2u\|_{n,1;\om} + C_\epsilon \|u\|_{n,1;\om}.
\end{equation*}
Similarly, by Lemmas \ref{conv}, \ref{ebd}, and \thref{GN2},
\begin{align*}
\|cu\|_{n,1;\om} &\le C \|c\|_{n,1;\om} \| u\|_{W^{1,n,1}(\om)} \\
& \le \frac{\epsilon}{2}\|D^2  u\|_{n,1;\om} + C_\epsilon \|u\|_{n,1;\om} ,
\end{align*}
where $C_\epsilon=C(n,\om, \|c\|_{n,1;\om},\epsilon)$.
This completes the proof of the lemma.
\end{proof}

 \begin{lem}\thlabel{estimate2} Assume that $b\in L^{n,1}(\om;\rn)$ and
$c\in L^{\frac{n}{2},1}(\om) \cap L^s(\om)$ for some $1<s<\frac{3}{2}$. Then for any $\epsilon>0$, there is a constant $C_\epsilon = C(n,\om, b,c,\epsilon)$ such that the following inequalities hold:
\begin{enumerate}[{\upshape (i)}]
\item For all $u \in W^{1,n,1}(\om)$,
$$
\|ub\|_{n,1;\om} + \|cu\|_{W^{-1,n,1}(\om)} \le \epsilon \|Du\|_{n,1;\om} + C_\epsilon \|u\|_{n,1;\om}.
$$

\item For all $v\in W^{1,\frac{n}{n-1},\infty}(\om)$,
$$\|b \cdot Dv\|_{W^{-1,\frac{n}{n-1},\infty}(\om)} + \|cv\|_{W^{-1,\frac{n}{n-1},\infty}(\om)} \le \epsilon \|Dv\|_{\frac{n}{n-1},\infty;\om} + C_\epsilon \|v\|_{\frac{n}{n-1},\infty;\om}.$$
\end{enumerate}
 \end{lem}
 \begin{proof}
Choose $b_k$ in $C_c^\infty(\om;\rn)$ and $c_k$ in $C_c^\infty(\om)$ such that $b_k \rightarrow b$ in $L^{n,1}(\om;\rn)$ and $c_k \rightarrow c$ in $L^{\frac{n}{2},1}(\om)\cap L^s(\om)$.
Let $\vp \in C_c^\infty(\om)$ be given.

(i) Suppose that  $u \in W^{1,n,1}(\om)$. Then by Lemmas \ref{conv} and \ref{ebd},
\begin{align*}
\|ub\|_{n,1;\om} &\le C  \|b-b_k\|_{n,1;\om} \|u\|_{\infty;\om} + C\|b_k\|_{\infty;\om}\|u\|_{n,1;\om} \\
 	& \le C\|b-b_k\|_{n,1;\om} \|Du\|_{n,1;\om} + C\left( \|b-b_k\|_{n,1;\om} + \|b_k\|_{\infty;\om}\right) \|u\|_{n,1;\om}.
\end{align*}
Moreover, if $n\ge3$, then
\begin{align*}
\left| \langle cu, \vp \rangle \right| &\le C  \|c-c_k\|_{\frac{n}{2},1;\om} \|u\|_{\infty;\om} \|\vp\|_{\frac{n}{n-2},\infty;\om}+ C\|c_k\|_{\infty;\om}\|u\|_{n,1;\om} \|\vp\|_{\frac{n}{n-1},\infty;\om} \\
 	& \le C\|c-c_k\|_{\frac{n}{2},1;\om} \|Du\|_{n,1;\om} \|\vp\|_{W^{1,\frac{n}{n-1},\infty}(\om)}\\
		& \qquad\qquad\qquad+ C\left( \|c-c_k\|_{\frac{n}{2},1;\om} + \|c_k\|_{\infty;\om}\right) \|u\|_{n,1;\om} \|\vp\|_{W^{1,\frac{n}{n-1},\infty}(\om)} .
\end{align*}
If $n=2$, then
\begin{align*}
\left| \langle cu, \vp \rangle \right| &\le C  \|c-c_k\|_{s;\om} \|u\|_{\infty;\om} \|\vp\|_{\frac{s}{s-1};\om}+ C\|c_k\|_{\infty;\om}\|u\|_{2,1;\om} \|\vp\|_{2,\infty;\om} \\
 	& \le C\|c-c_k\|_{s;\om} \|Du\|_{2,1;\om} \|\vp\|_{W^{1,2,\infty}(\om)}\\
 &\qquad + C\left( \|c-c_k\|_{s;\om} + \|c_k\|_{\infty;\om}\right) \|u\|_{2,1;\om} \|\vp\|_{W^{1,2,\infty}(\om)} .
\end{align*}
Hence, given $\epsilon>0$, we obtain the desired inequality by choosing $k$ sufficiently large so that
$$
C \|b-b_k\|_{n,1;\om} + C\|c-c_k\|_{\frac{n}{2},1;\om} + C\|c-c_k\|_{s;\om}\le \frac{\epsilon}{2}.
$$

\medskip
(ii) Suppose that  $v\in W^{1,\frac{n}{n-1},\infty}(\om)$. Then the integration by parts gives
\begin{align*}
\int_\om \vp b \cdot Dv \,dx  & = \int_\om \vp (b-b_k) \cdot Dv \,dx + \int_\om \vp b_k \cdot Dv  \,dx \\
	& = \int_\om \vp (b-b_k) \cdot Dv \,dx -\int_\om v \vp \,\div b_k \,dx   - \int_\om v b_k \cdot D\vp \,dx.
\end{align*}
Thus by Lemmas \ref{holder} and \ref{ebd},
\begin{align*}
\left| \langle b \cdot Dv , \vp \rangle \right| & \le  C\|b-b_k\|_{n,1;\om} \|Dv\|_{\frac{n}{n-1},\infty;\om} \|\vp\|_{W^{1,n,1}(\om)} \\
	& \qquad \qquad + C \left( \|\div b_k\|_{\infty;\om}  + \|b_k\|_{\infty;\om}\right) \|v\|_{\frac{n}{n-1},\infty;\om} \|\vp\|_{W^{1,n,1}(\om)}.
\end{align*}
On the other hand, writing
$$
 \int_\om cv \vp \,dx  = \int_\om  (c-c_k)v\vp \,dx + \int_\om c_k v \vp  \,dx,
$$
and using Lemmas \ref{holder} and \ref{ebd} again, we see that if $n\ge3$, then
\begin{align*}
\left| \langle cv , \vp \rangle \right| & \le C \|c-c_k\|_{\frac{n}{2},1;\om} \|v\|_{\frac{n}{n-2},\infty;\om}\|\vp\|_{\infty;\om} + C \|c_k\|_{\infty;\om} \|v\|_{\frac{n}{n-1},\infty;\om} \|\vp\|_{n,1;\om} \\
	&\le C \|c - c_k \|_{\frac{n}{2},1;\om} \|v\|_{W^{1,\frac{n}{n-1},\infty}(\om)} \|\vp\|_{W^{1,n,1}(\om)}  + C \|c_k\|_{\infty;\om} \|v\|_{\frac{n}{n-1},\infty;\om} \|\vp\|_{n,1;\om}\\
	&\le C\|c - c_k \|_{\frac{n}{2},1;\om} \|Dv\|_{\frac{n}{n-1},\infty;\om} \|\vp\|_{W^{1,n,1}(\om)} \\
		& \qquad \qquad \qquad \quad + C\left(\|c-c_k\|_{\frac{n}{2},1;\om} + \|c_k\|_{\infty;\om}\right) \|v\|_{\frac{n}{n-1},\infty;\om} \|\vp\|_{W^{1,n,1}(\om)}.
\end{align*}
Similarly, if $n=2$, then
\begin{align*}
\left| \langle cv , \vp \rangle \right| & \le C \|c-c_k\|_{s;\om} \|v\|_{\frac{s}{s-1};\om}\|\vp\|_{\infty;\om} + C \|c_k\|_{\infty;\om} \|v\|_{2,\infty;\om} \|\vp\|_{2,1;\om} \\
	&\le C\|c-c_k\|_{s;\om}\|v\|_{W^{1,2,\infty}(\om)}\|\vp\|_{W^{1,2,1}(\om)} + C \|c_k\|_{\infty;\om} \|v\|_{2,\infty;\om} \|\vp\|_{W^{1,2,1}(\om)}\\
	&\le C \|c-c_k\|_{s;\om} \|Dv\|_{2,\infty;\om} + C \left( \|c-c_k\|_{s;\om} + \|c_k\|_{\infty;\om}\right)\|v\|_{2,\infty;\om} \|\vp\|_{W^{1,2,1}(\om)}.
\end{align*}
Hence the desired estimate is proved by choosing a sufficiently large $k$.
 \end{proof}

\section{Proof of \thref{thm1}}
To prove \thref{thm1} (i), we need the following a priori estimate for the distribution functions of  weak solutions in $W^{1,2}_0(\om) \cap L^\infty(\om)$ of $\eqref{Eq3}$.

\begin{lem}\thlabel{apriori est-2-1}
Let $\om$ be a bounded Lipschitz domain in $\mathbb{R}^n , n \ge 2$.
Suppose that $b\in L^{2}(\om;\rn)$, $c\in L^{1}(\om)$, and $c\ge0$ in $\Omega$.
Then there is  a constant $C=C(n,\Omega , \delta )$ such that if $u \in  W^{1,2}_0(\om) \cap L^\infty(\om)$
satisfies
\begin{equation}\label{thm1-I-1}
\int_\om \left( A D u \cdot D \vp -u b \cdot  D \vp  + cu \vp\right) dx =   \langle f ,  \vp \rangle
\end{equation}
for all $\vp \in C^\infty_c(\om)$, where $f \in W^{-1,2}(\om)$,   then
\[
d_u (\alpha ) \le \frac{C \left( \|b\|_{2;\Omega}+ \|f\|_{W^{-1,2}(\om)} \right)^2 }{[\ln (1+\alpha )]^2}
\quad\mbox{for all}\,\, \alpha >0.
\]
\end{lem}

\begin{proof}
Suppose that $u \in W_0^{1,2}(\om) \cap L^\infty(\om)$ satisfies \eqref{thm1-I-1} holds for all $\vp \in C^\infty_c(\om)$. By Lemma \ref{dense}, it can be easily shown that \eqref{thm1-I-1} holds for all $\vp \in W^{1,2}_0(\om) \cap L^\infty(\om)$.
Hence taking $\vp = \frac{u}{1+|u|}$ in \eqref{thm1-I-1} and using the Poincar\'{e} inequality, we have
\begin{align*}
\delta \int_\Omega \frac{|Du|^2}{(1+|u|)^2} \, dx &\le \int_\Omega A \left(\frac{Du}{1+|u|}\right) \left(\frac{Du}{1+|u|}\right) \, dx  \\
& = \int_\Omega   \frac{u b \cdot Du}{(1+|u|)^2}  \, dx - \int_\Omega   \frac{c u^2}{1+|u|}  \, dx + \left\langle  f ,  \frac{ u}{1+|u|}  \right\rangle\\
& \le  \int_\Omega   |b| \frac{ | Du| }{1+|u|}  \, dx + C \|f\|_{W^{-1,2}(\om)} \left\| \frac{Du}{1+|u|} \right\|_{2;\Omega}
\end{align*}
and so
\[
\delta \left\| \frac{Du}{1+|u|} \right\|_{2;\Omega} \le \|b\|_{2;\Omega}+ C \|f\|_{W^{-1,2}(\om)}.
\]
Since
$$
\ln (1+|u|) \in W_0^{1,2}(\Omega ) \quad\mbox{and}\quad \left|D [\ln (1+|u|) ]\right|= \frac{|Du|}{1+|u|},
$$
it follows from the Poincar\'{e} inequality that
\[
  \left\| \ln (1+|u|) \right\|_{2;\Omega} \le C \left( \|b\|_{2;\Omega}+ \|f\|_{W^{-1,2}(\om)} \right)
\]
for some $C=C(n, \Omega , \delta )$.
Therefore, for $\alpha >0$,
\begin{align*}
d_u (\alpha )  & = \left| \{x \in \Omega \, :\, \ln (1+|u(x)|) > \ln (1+\alpha ) \} \right|\\
& \le \int_{\Omega} \left[\frac{\ln (1+|u|)}{\ln (1+\alpha )}\right]^2 \, dx \\
& \le \frac{C^2 \left( \|b\|_{2;\Omega}+ \|f\|_{W^{-1,2}(\om)} \right)^2 }{[\ln (1+\alpha )]^2} .
\end{align*}
This completes the proof of the lemma.
\end{proof}

\begin{proof}[Proof of \thref{thm1} \textup{(i)}]

For each  $u\in W^{1,n,1}_0(\om)$, we define
$$
L_0 u =-{\rm div} \left( A D  u \right) \quad\mbox{and}\quad L_1 u  =L_0 u + \div(ub) + cu .
$$
Then it follows from Corollary \ref{kimtsai2} and Lemma  \ref{estimate2} that $L_0$ and $L_1$ are   bounded linear operators from $W^{1,n,1}_0(\om)$ to $W^{-1,n,1}(\om)$ and $L_0 : W^{1,n,1}_0(\om) \to W^{-1,n,1}(\om)$ is an isomorphism.
Hence to prove  \thref{thm1} \textup{(i)}, it suffices, by the method of continuity, to prove the a priori estimate:
\begin{equation}\label{aprori-cont}
\|u\|_{W^{1,n,1}(\om)} \le C \| (1-t)L_0 u +t L_1 u \|_{W^{-1,n,1}(\om)}
\end{equation}
for all $u \in W_0^{1,n,1}(\om)$ and $t \in [0,1]$, where $C$ is a constant depending only on $n$, $\Omega$, $A$, $b$, and $c$.

Suppose that  $u \in W_0^{1,n,1}(\om)$, $t \in [0,1]$, and   $f = (1-t)L_0 u +t L_1 u$. By linearity, we may assume that  $ \| f\|_{W^{-1,n,1}(\om)} \le 1$.
Then by \thref{kimtsai2} and \thref{estimate2},
\begin{align*}
\|u\|_{W^{1,n,1}(\om)} & \le C \| L_0 u \|_{W^{-1,n,1}(\om)} \\
& =  C \| f  - t \, \div(ub) -t  cu \|_{W^{-1,n,1}(\om)} \\
	&  \le C \| f\|_{W^{-1,n,1}(\om)} + C \| ub \|_{n,1;\om} + C \|cu\|_{W^{-1,n,1}(\om)} \\
	&\le C    + \frac{1}{4}\|Du\|_{n,1;\om} + C\|u\|_{n,1;\om}.
	\end{align*}
Moreover, by Lemmas \ref{itp-lorentz} and \ref{ebd},
\[
C\|u\|_{n,1;\om}
 \le C \|u\|_{1;\Omega}^{\frac{1}{n}}\|u\|_{W^{1,n,1}(\om)}^{1-\frac{1}{n}}
  \le \frac{1}{4}\|u\|_{W^{1,n,1}(\om)} + C \|u\|_{1;\Omega}.
\]
For $\alpha >0$, let $E_\alpha =\{x\in \Omega \,|\, |u(x)| > \alpha \}$. Then since
\[
 \|f\|_{W^{-1,2}(\om)}  \le C  \|f\|_{W^{-1,n}(\om)}  \le C ,
\]
it follows from Lemma \ref{apriori est-2-1} that
\begin{align*}
C \|u\|_{1;\Omega} & =  C \int_{\Omega \setminus E_\alpha} |u|\, dx  + C \int_{E_\alpha } |u| \, dx \\
&\le   C \alpha +  C d_u (\alpha )^{1/2} \|u\|_{2;\Omega }  \\
& \le C \alpha + \frac{C \left(   \|b\|_{2;\Omega}+ \|f\|_{W^{-1,2}(\om)} \right)}{\ln (1+\alpha )} \|u\|_{W^{1,n,1}(\om)}\\
&\le  C \alpha + \frac{1}{4} \|u\|_{W^{1,n,1}(\om)}
\end{align*}
for some large $\alpha$ depending only on $n,\Omega, A, b$, and $c$. Combining the above three  estimates, we obtain
\[
\|u\|_{W^{1,n,1}(\om)} \le C +C \alpha + \frac{3}{4} \|u\|_{W^{1,n,1}(\om)},
\]
which  proves   the a priori estimate (\ref{aprori-cont}).
The proof of \thref{thm1} \textup{(i)} is complete.
\end{proof}

\begin{proof}[Proof of \thref{thm1} \textup{(ii)}]
By the method of continuity, it suffices to show that there is a constant $C= C(n, \Omega , A , b,c)$ such that
\begin{equation}\label{aprori-cont-dual}
\|v\|_{W^{1,n',\infty}(\om)} \le C \| -\div (A Dv ) - t b \cdot D  v + t cv \|_{W^{-1,n',\infty}_0(\om)}
\end{equation}
for all $v \in W^{1,n',\infty}_0(\om)$ and $t \in [0,1]$, where $n' = \frac{n}{n-1}$.

Let $v \in W^{1,n',\infty}_0(\om)$ and $t \in [0,1]$ be given. Define
\[
g= -\div (A Dv ) - t b \cdot D  v + t cv .
\]
Then  for all $\psi \in C_c^\infty (\om)$, we have
\begin{equation}\label{thm1-I-5}
\int_\om \left( A D v \cdot D \psi - t \psi b \cdot D v + t cv \psi \right) dx =  \langle g , \psi \rangle .
\end{equation}
It follows from Lemma \ref{estimate2} that
$b\cdot D v$, $cv$, and $g$ all belong to $W^{-1,n', \infty}(\Omega )$. Since
$W^{-1,n', \infty}(\Omega )$ is the dual space of $W_0^{1,n,1}(\Omega )$, it easily follows from  Lemma \ref{dense} that (\ref{thm1-I-5}) holds for all $\psi \in W^{1,n,1}_0(\om)$.

Fix any $F  \in L^{n,1}(\om;\rn)$. Then by the proof of \thref{thm1} (i),
  there exists a unique $u\in W^{1,n,1}_0(\om)$ such that
$$
\|u\|_{W^{1,n,1}(\om)} \le C \|F\|_{n,1;\om}
$$
and
\begin{equation}\label{thm1-I-6}
\int_\om \left(A Du \cdot D \vp  - t u b \cdot D \vp + t cu\vp\right) dx = - \int_\om F \cdot D\vp \,dx
\end{equation}
for all $\vp\in C_c^\infty (\om)$.

We claim that \eqref{thm1-I-6} also holds for all  $\vp \in W^{1,n',\infty}_0(\om)$.
To show this, we  choose functions $f_k \in C_c^\infty(\om)$ and $H_k \in  C_c^\infty(\om;\rn)$   such that $f_k \rightarrow -t cu$ in $L^{\frac{n}{2},1}(\om)\cap L^s(\om)$ and $H_k \rightarrow F - tub$ in $L^{n,1}(\om;\rn)$.
Choose any $n<r<\infty$.
Then by  \thref{kimtsai2},   there is a unique $u_k\in W^{1,r}_0(\om)$ such that
\begin{equation}\label{thm1-I-9}
\int_\om A D  u_k \cdot D \vp \,dx = \int_\om f_k \vp  \,dx - \int_\om H_k \cdot D\vp \,dx
\end{equation}
for all $\vp \in W^{1,r'}_0(\om)$. Since $n' >r'$, we have  $ W^{1,n',\infty}_0(\om) \hookrightarrow  W^{1,r'}_0(\om)$. Hence \eqref{thm1-I-9} also holds for all $\vp\in W^{1,n',\infty}_0(\om)$. On one hand, it follows from \thref{ebd} that if $\vp \in W^{1,n',\infty}_0(\om)$, then    $\vp \in L^{\frac{n}{n-2},\infty }(\om)$ for $n \ge 3$ and $\vp \in  L^{s'}(\Omega )$ for $n=2$.
On the other hand, by \thref{kimtsai2}, we have
\begin{align*}
\|u_k - u\|_{W^{1,n,1}(\om)} &\le C \| \div(F - t ub  -   H_k ) - t cu - f_k \|_{W^{-1,n,1}(\om)} \\
& \le
\left\{
\begin{aligned}
&C \|F- tub - H_k \|_{n,1;\om} + C\|tc u + f_k\|_{\frac{n}{2},1;\om} &\quad\text{if $n\ge3$},\\
&C \|F- t ub - H_k \|_{n,1;\om} + C\|t cu + f_k\|_{s;\om}&\quad\text{if $n=2$}.
\end{aligned}
\right.
\end{align*}
Therefore, letting $k\rightarrow \infty$ in \eqref{thm1-I-9}, we see that \eqref{thm1-I-6} holds for all $\vp \in W^{1,n',\infty}_0(\om)$.

Now, putting $\psi =u$ and $\vp = v$ in \eqref{thm1-I-5} and \eqref{thm1-I-6}, respectively, we obtain
\begin{align*}
\left|\int_\om F \cdot Dv \, dx \right|= \left|\langle g ,  u  \rangle \right| &\le C \|g\|_{W^{-1,n',\infty}_0(\om)} \| u\|_{W^{1,n ,1}_0(\om)} \\
& \le C \|g\|_{W^{-1,n',\infty}_0(\om)} \|F\|_{n,1;\om}.
\end{align*}
Since $F \in L^{n,1}(\om;\rn)$ is arbitrary, it follows from   \thref{dual-lorentz} (the duality theorem) that
\[
\|Dv\|_{n',\infty;\om} \le C \|g\|_{W^{-1,n',\infty}_0(\om)}.
\]
Moreover, by the Sobolev inequality, we have
\[
\|v\|_{n',\infty;\om} \le C \|v\|_{(r')^*;\om} \le C \|Dv\|_{r';\om} \le C \|Dv\|_{n',\infty;\om},
\]
which proves the a priori estimate (\ref{aprori-cont-dual}). The proof of \thref{thm1} (ii) is complete.
\end{proof}

\section{Proof of \thref{thm1-2}}

Let $L_0 '$ and $L'$ be  differential operators defined by
$$
L_0 'u=-\sum_{i,j=1}^n a^{ij}D_{ij} u \quad\mbox{and}\quad L ' u = L_0 ' u + b \cdot  D  u+ cu.
$$
Recall the famous  Aleksandrov  maximum principle (see, e.g., \cite[Theorem 9.1]{gilbarg}) for elliptic equations in non-divergence form.

\begin{thm}\thlabel{aleksandrov}
Let $\om$ be a bounded domain in $\rn$.
Suppose that $b\in L^n(\om; \rn)$, $c\ge 0$ in $\om$, and $f\in L^n(\om)$.
If $u\in W^{2,n}_{loc}(\om)\cap C(\overline \om)$ satisfies $L'u \le f$ in $\om$, then
$$
\sup_\om u \le \sup_{\partial \om} u^+ + C \|f\|_{n;\om},
$$
where $C=C(n, \textup{diam} \,\om, \delta, \|b\|_{n;\om})$.
\end{thm}

\begin{proof}[Proof of \thref{thm1-2} \textup{(i)}]
Given $u\in W^{2,n,1}(\om)\cap W^{1,n,1}_0(\om)$  and $t \in [0,1]$, we write
\[
f= L_0 ' u +t  b \cdot  D  u +t cu .
\]
Then by Lemmas \ref{kimtsai} and  \ref{estimate1}, we have
\begin{align*}
\| u\|_{W^{2,n,1}(\om)} &\le C \|L_0 'u \|_{n,1;\om}   \\
&\le C \|f \|_{n,1;\om} + C \|t  b \cdot  D  u +t  cu\|_{n,1;\om}  \\
	&\le C \|f\|_{n,1;\om} + \frac{1}{2}\|D^2u\|_{n,1;\om} + C \|u\|_{n,1;\om}  .
\end{align*}
Moreover, since $W^{1,n,1} (\om) \hookrightarrow C(\overline \om)$, it follows from Theorem  \ref{aleksandrov}  that
\[
\|u\|_{n,1;\om} \le C \| u \|_{\infty;\om}  \le    C \|f\|_{n;\om} .
\]
Combining the two estimates, we obtain the following  a priori estimate:
\begin{equation}\label{aprori-cont-strong}
\|u\|_{W^{2,n,1}(\om)} \le C \|  L_0 ' u +t  b \cdot  D  u + t cu  \|_{L^{n,1}(\om)}
\end{equation}
for all $u\in  W^{2,n,1}(\om)\cap W^{1,n,1}_0(\om)$ and $t \in [0,1]$, where $C = C(n, \Omega , A, b, \|c\|_{n,1;\om})$. The proof of \thref{thm1-2} \textup{(i)}
is complete by the method of continuity.
\end{proof}

\begin{proof}[Proof of \thref{thm1-2} \textup{(ii)}]
By \thref{thm1-2} \textup{(i)}, $L'$ is an isomorphism from $W^{2,n,1}(\om)\cap W^{1,n,1}_0(\om)$ onto $L^{n,1}(\om)$. Let $\mathcal{S}: L^{n,1}(\om)\rightarrow W^{2,n,1}(\om)\cap W^{1,n,1}_0(\om)$ be the inverse   of $L'$.
Then for all $f \in L^{n,1}(\om)$,
\[
  \|D\mathcal{S}(f)\|_{\infty;\om} \le C  \|D\mathcal{S}(f)\|_{W^{1,n,1}(\om)}  \le C   \|f\|_{n,1;\om}.
\]
Hence given  $G\in L^1(\om;\rn)$, the mapping
$$
f\in L^{n,1}(\om) \mapsto T(f)= -\int_\om G \cdot D \mathcal{S}(f) \,dx
$$
defines a bounded linear functional $T$  on $L^{n,1}(\om)$, which satisfies
$$
\|T\|_{L^{n,1}(\om)^*} \le C \|G\|_{1;\om}.
$$
Hence by \thref{dual-lorentz}, there exists  a unique $v \in L^{n',\infty}(\om)$ such that
\[
 \|v\|_{n',\infty;\om}\le C \|T\|_{L^{n,1}(\om)^*}\le C \|G\|_{1;\om}
\]
and
\begin{equation}\label{VWW-via S0}
\int_\om v f \,dx    = -\int_\om G \cdot D \mathcal{S}(f) \,dx \quad \text{for all $f\in L^{n,1}(\om)$.}
\end{equation}
Since $\mathcal{S} $ is the inverse   of the isomorphism $L'$, it immediately  follows from (\ref{VWW-via S0}) that
\begin{equation}\label{VWW-via S}
\int_\om v L' \psi \,dx    = -\int_\om G \cdot D \psi \,dx
\end{equation}
for all $ \psi \in W^{2,n,1}(\om)\cap W^{1,n,1}_0(\om)$. Therefore, $v   $ is a very weak solution in $L^{n',\infty}(\om)$ of (\ref{Eq4}) satisfying the desired estimate.

To complete the proof of of \thref{thm1-2} \textup{(ii)}, it remains to prove the uniqueness assertion. Let $v \in L^{n',\infty}(\om)$ is a very weak solution of (\ref{Eq4}) with the trivial data $G=0$. Then for all $\psi \in C^{1,1}(\overline{\om})$ with $\psi =0$ on $\partial\Omega$, we have
\[
\int_\om v L' \psi \,dx    = 0.
\]
Let $f\in L^{n,1}(\Omega )$ and $u =\mathcal{S}(f)$. Then by Lemma \ref{dense}, there is a sequence $\{\psi_k\}$ in $C^{1,1}(\overline{\om})$ such that $ \psi_k  =0$ on $\partial\Omega$ and  $\psi_k \rightarrow u$ in $W^{2,n,1}(\om)$.
Since $L'\psi_k \to L' u =f$ in $L^{n,1}(\om )$, we have
\[
\int_\om v f \,dx    = \lim_{k\to\infty} \int_\om v L' \psi_k  \,dx =0.
\]
Since $f\in L^{n,1}(\Omega )$ is arbitrary, we deduce that $v =0$ identically on $\Omega$. The proof of of \thref{thm1-2} \textup{(ii)} is complete.
\end{proof}

\bibliographystyle{amsplain}

\begin{thebibliography}{10}

\bibitem{adams}
R.~A. Adams and J.~J.~F. Fournier, \emph{Real interpolation of {S}obolev spaces
  on subdomains of $\rn$}, Canadian J. Math. \textbf{30} (1978), no.~1,
  190--214.

\bibitem{AQ}
P.~Auscher and M.~Qafsaoui, \emph{Observations on ${W}^{1,p}$ estimates for
  divergence elliptic equations with ${VMO}$ coefficients}, Boll. Unione Mat.
  Ital. Sez. B Artic. Ric. Mat. (8) \textbf{5} (2002), no.~2, 487--509.

\bibitem{bennett}
C.~Bennett and R.~Sharpley, \emph{Interpolation of {O}perators}, Academic
  Press, Boston, MA, 1988.

\bibitem{bergh}
J.~Bergh and J.~L\"ofstr\"om, \emph{Interpolation {S}paces. {A}n
  {I}ntroduction}, Springer-Verlag, Berlin-New York, 1976.


\bibitem{itp}
J.~Byeon, H.~Kim, and J.~Oh, \emph{Interpolation inequalities in function
  spaces of Sobolev-Lorentz type},  J. Math. Anal. Appl. \textbf{516} (2022), no. ~2, Paper No. 126519.
\bibitem{BS}
S.~Byun, \emph{Elliptic equations with ${BMO}$ coefficients in {L}ipschitz
  domains}, Trans. Amer. Math. Soc. \textbf{357} (2005), no.~3, 1025--1046.

\bibitem{CaTr}
P.~Cavaliere and M.~Transirico, \emph{The Dirichlet problem for elliptic equations in the plane},  Comment. Math. Univ. Carolin. \textbf{46} (2005), no.~4, 751--758.


\bibitem{CFL}
F.~Chiarenza, M.~Frasca, and P.~Longo, \emph{${W}^{2,p}$-solvability of the
  {D}irichlet problem for nondivergence elliptic equations with ${VMO}$
  coefficients}, Trans. Amer. Math. Soc. \textbf{336} (1993), no.~2, 841--853.

\bibitem{devore}
R.~DeVore and K.~Scherer, \emph{Interpolation of linear operators on {S}obolev
  spaces}, Ann. of Math (2) \textbf{109} (1979), no.~3, 583--599.

\bibitem{DK0}
H.~Dong and D.~Kim, \emph{Elliptic equations in divergence form with partially BMO coefficients}, Arich. Ration. Mech. Anal. \textbf{196} (2010), no.~1, 25--70

\bibitem{DK}
H.~Dong and D.~Kim, \emph{On the ${L}_p$-solvability of higher order parabolic
  and elliptic systems with BMO coefficients}, Arch. Ration. Mech. Anal.
  \textbf{199} (2011), no.~3, 889--941.

\bibitem{evans}
L.~C. Evans, \emph{Partial {D}ifferential {E}quations}, Second edition, American
  Mathematical Society, Providence, RI, 2010.

\bibitem{gerhardt}
C.~Gerhardt, \emph{Stationary solutions to the Navier-Stokes equations in dimension four}, Math.
Z. \textbf{165} (1979), no.~2, 193--197.

\bibitem{gilbarg}
D.~Gilbarg and N.~S. Trudinger, \emph{Elliptic {P}artial {D}ifferential
  {E}quations of {S}econd {O}rder}, Reprint of the 1998 edition edition,
  Springer-Verlag, Berlin, 2001.

\bibitem{Gra}
L.~Grafakos,\emph{Classical Fourier Analysis}, Third edition, Graduate Texts in Mathematics, 249. Springer, New York, 2014.

\bibitem{kimkang}
B.~Kang and H.~Kim,  \emph{$W^{1,p}$-estimates for elliptic equations with lower order terms},  Commun. Pure Appl. Anal. \textbf{16} (2017), no.~3, 799--821.

\bibitem{kimkim}
H.~Kim and Y.-H.~Kim, \emph{On weak solutions of elliptic equations with singular
drifts}, SIAM J. Math. Anal. \textbf{47} (2015), no.~2, 1271--1290.

\bibitem{kimkwon}
H.~Kim and H.~Kwon, \emph{Dirichlet and Neumann problems for elliptic equations with singular drifts on Lipschitz domains},
Trans. Amer. Math. Soc. \textbf{375} (2022), no.~9, 6537--6574.


\bibitem{kim}
H.~Kim and T.-P. Tsai, \emph{Existence, uniqueness, and regularity results for
  elliptic equations with drift terms in critical weak spaces}, SIAM J. Math.
  Anal. \textbf{52} (2020), no.~2, 1146--1191.

\bibitem{krylov0}
N.~V. Krylov, \emph{Lectures on elliptic and parabolic equations in Sobolev spaces}, Graduate Studies in Mathematics, 96. American Mathematical Society, Providence, RI, 2008.

\bibitem{krylov}
N.~V. Krylov, \emph{Elliptic equations with ${VMO}$ $a$, $b\in {L}_d$, and
  $c\in {L}_{d/2}$}, Trans. Amer. Math. Soc. \textbf{374} (2021), no.~4,
  2805--2822.

\bibitem{kwon0}
H. ~Kwon, \emph{Existence and uniqueness of weak solution in $W^{1,2+\varepsilon}$ for elliptic equations with drifts in weak-$L^n$ spaces},  J. Math. Anal. Appl. \textbf{500} (2021), no. 1, Paper No. 125165.

\bibitem{kwon}
H. ~Kwon, \emph{Elliptic equations in divergence form with drifts in $L^2$},  Proc. Amer. Math. Soc. \textbf{150} (2022), no. 8, 3415--3429.



\bibitem{oneil}
R.~O'Neil, \emph{Convolution operators and ${L}(p,q)$ spaces}, Duke Math. J.
  \textbf{30} (1963), 129--142.

\bibitem{stein}
E.~M. Stein, \emph{Singular {I}ntegrals and {D}ifferentiability {P}roperties of
  {F}unctions}, Princeton University Press, Princeton, N.J., 1970.


\end{thebibliography}

\end{document}